\documentclass[12pt,a4paper]{amsart}
\usepackage{amssymb,amsmath,amsthm}
\usepackage{graphicx}
\usepackage{hyperref}
 \usepackage[color=green!40]{todonotes}

\newcommand\cC{{\mathcal C}}

\newcommand\cF{{\mathcal F}}

\newcommand\cN{{\mathcal N}}

\newcommand\cT{{\mathcal T}}

\makeatletter
\newtheorem*{rep@theorem}{\rep@title}
\newcommand{\newreptheorem}[2]{%
\newenvironment{rep#1}[1]{%
 \def\rep@title{#2 \ref{##1}}%
 \begin{rep@theorem}}%
 {\end{rep@theorem}}}
\makeatother

\theoremstyle{plain}
\newtheorem{theorem}{Theorem}[section]
\newreptheorem{theorem}{Theorem}
\newtheorem{lemma}[theorem]{Lemma}

\newtheorem{conjecture}[theorem]{Conjecture}
\newtheorem{proposition}[theorem]{Proposition}

\theoremstyle{definition}

\newtheorem{claim}[theorem]{Claim}

\newcommand\cref[1]{Corollary~\ref{cor:#1}}

\title{Triangles in intersecting families}
\author{D\'aniel T. Nagy}
\address{Alfr\'ed R\'enyi Institute of Mathematics}
\email{nagydani@renyi.hu}
\thanks{Nagy's research is partially supported by NKFIH grants FK 132060 and PD 137779 and by the J\'anos Bolyai Research Fellowship of the Hungarian Academy of Sciences}

\author{Bal\'azs Patk\'os}
\address{Alfr\'ed R\'enyi Institute of Mathematics and Moscow Institute of Physics and Technology}
\email{patkos@renyi.hu}
\thanks{Patk\'os's research is partially supported by NKFIH grants SNN 129364 and FK 132060.}

\date{}

\begin{document}

\maketitle

\begin{abstract}
    We prove the following the generalized Tur\'an type result. A collection $\cT$ of $r$ sets is an $r$-triangle if for every $T_1,T_2,\dots,T_{r-1}\in \cT$ we have $\cap_{i=1}^{r-1}T_i\neq\emptyset$, but $\cap_{T\in \cT}T$ is empty. A family $\cF$ of sets is $r$-wise intersecting if for any $F_1,F_2,\dots,F_r\in \cF$ we have $\cap_{i=1}^rF_i\neq \emptyset$ or equivalently if $\cF$ does not contain any $m$-triangle for $m=2,3,\dots,r$. We prove that if $n\ge n_0(r,k)$, then the $r$-wise intersecting family $\cF\subseteq \binom{[n]}{k}$ containing the most number of $(r+1)$-triangles is isomorphic to $\{F\in \binom{[n]}{k}:|F\cap [r+1]|\ge r\}$.
\end{abstract}

\section{Introduction}
Tur\'an type problems form one of the most studied areas in extremal combinatorics. They ask for the maximum size of a combinatorial structure that avoids some forbidden substructure. Most common examples are the following: the maximum number of edges in an $H$-free graph on $n$ vertices, the maximum number of sets in $2^{[n]}$ that avoids some inclusion pattern $P$, the maximum number of sets in in $2^{[n]}$ satisfying some intersection property. Recently, so called \textit{generalized} Tur\'an problems have attracted the attention of researchers, especially in the domain of extremal graph theory. Given a forbidden graph $H$, what is the maximum number of copies of a fixed graph $F$ among $n$-vertex $H$-free graphs $G$? A rapidly growing literature addresses problems of this type.

There exist some results of similar flavor concerning uniform intersecting families \cite{FK, FKK1, FKK2} or  the union of two intersecting families \cite{P}. In this short note, we consider $k$-uniform intersecting families $\cF\subseteq \binom{[n]}{k}:=\{F\subseteq \{1,2,\dots,n\}:|F|=k\}$. By definition, if $\cF$ is intersecting, then it cannot contain a pair of disjoint sets. But what is the maximum number of triples $F,F',F''$ that a family $\cF\subseteq \binom{[n]}{k}$ can contain with $F\cap F' \cap F''=\emptyset$? Such triples are called \textit{triangles}. A natural and well-known candidate of an intersecting family with many triangles is the following. For any 3-subset $X$ of $[n]$, we define $\cF_{X,k}=\{F\in \binom{[n]}{k}: |F\cap X|\ge 2\}$ and $\cF'_{X,k}=\{F\in \binom{[n]}{k}: |F\cap X|= 2\}$. We write $n_{3,k}$ to denote the number of triangles in $\cF_{X,k}$. 

\begin{theorem}\label{tria}
For every $k\ge 2$ if $\cF\subseteq \binom{[n]}{k}$ is intersecting with $n\ge 4k^6$, then the number of triangles in $\cF$ is at most $n_{3,k}$ and equality holds if and only if $\cF'_{X,k}\subseteq \cF\subseteq \cF_{X,k}$ for some 3-subset $X$ of $[n]$.
\end{theorem}

More generally, one can consider \textit{$r$-wise intersecting families}. $\cF$ has this property if for any $F_1,F_2,\dots, F_r\in \cF$, we have $\cap_{i=1}^rF_i\neq \emptyset$. As proved by Frankl \cite{F}, the maximum size of an $r$-wise intersecting family $\cF\subseteq \binom{[n]}{k}$ is $\binom{n-1}{k-1}$ whenever $\frac{r-1}{r}n\ge k$ holds. The extremal family is unique unless $r=2$ and $n=2k$: all $k$-sets containing a fixed element of the ground set. The $r$-wise intersecting property can be formulated via forbidden configurations. A family $\cT$ of $k$-sets with $\cap_{T\in \cT}T=\emptyset$ is an $r$-triangle if $|\cT|=r$ and for any $T_1,T_2,\dots,T_{r-1}\in\cT$ we have $\cap_{i=1}^{r-1}T_i\neq \emptyset$. Let $\Delta_{r,k}$ denote the set of all $k$-uniform $r$-triangles. Then $\cF\subseteq \binom{[n]}{k}$ is $r$-wise intersecting if and only if it is $\Delta_{m,k}$-free for all $2\le m \le r$. Therefore, a natural generalization of Theorem \ref{tria} would be to maximize the number of $(r+1)$-triangles in $r$-wise intersecting families.

For any $X\subset [n]$ and integer $k$ we define $\cF_{X,k}=\{F\in \binom{[n]}{k}:|F\cap X|\ge |X|-1\}$ and $\cF'_{X,k}=\{F\in \binom{[n]}{k}:|F\cap X|= |X|-1\}$. Finally, we write $n_{r+1,k}$ to denote the number of $(r+1)$-triangles in $\cF_{X,k}$ with $|X|=r+1$.

\begin{theorem}\label{gener}
For every $k\ge r\ge 3$  if $\cF\subseteq \binom{[n]}{k}$ is $r$-wise intersecting with $n\ge 4k^{r(r+1)}$, then the number of $(r+1)$-triangles in $\cF$ is at most $n_{r+1,k}$ and equality holds if and only if $\cF'_{X,k}\subseteq\cF\subseteq\cF_{X,k}$ for some $(r+1)$-subset $X$ of $[n]$.
\end{theorem}

It is not hard to see and  we will show after Proposition \ref{rwisetau}, if $r>k$, then every $r$-wise intersecting $k$-uniform family $\cF$ is trivial, i.e. all sets in $\cF$ share a common element and thus $\cF$ does not contain any $\ell$-triangles for any $\ell\ge 2$. Therefore Theorem \ref{gener} covers all interesting cases.

\section{Preliminaries}
 
 In this section, we gather some easy statements that we will need during the proofs. 

\begin{proposition}\label{ordmag}
For any fixed $2\le r\le k$ and $n\ge k^2$ we have $$n_{r+1,k}\ge \frac{1}{2}\binom{n-r-1}{k-r}^{r+1}\ge \frac{(n-1)^{r+1}}{2k^{r+1}}.$$
\end{proposition}

\begin{proof}
As any $(r+1)$-triangle in $\cF_{X,k}$ must contain one member of $\cF_x=\{F\in \cF_{X,k}:x\notin F\}$ for each $x\in X$ and these sets should not have a common element outside $X$, by exclusion-inclusion, we obtain the following lower bound on $n_{r+1,k}$:
$$\binom{n-r-1}{k-r}^{r+1}-(n-r-1)\binom{n-r-2}{k-r-1}^{r+1}\ge \binom{n-r-1}{k-r}^{r+1}\left[1-\frac{k^{r+1}}{n^r}\right] \ge \frac{1}{2}\binom{n-r-1}{k-r}^{r+1}.$$
\end{proof}

The \textit{covering number} $\tau(\cF)$ of a family $\cF$ of sets is the smallest size of a \textit{cover} set $X$ with $X\cap F\neq \emptyset$ for all $F\in \cF$. The family of covers of $\cF$ is denoted by $\cC_\cF$. 

\begin{proposition}\label{rwisetau}
If $\cF$ is $r$-wise intersecting with $\tau(\cF)=t$, then $\cF$ is $[(r-2)(t-1)+1]$-intersecting.
\end{proposition}

\begin{proof}
Suppose not and let $F,F'\in \cF$ be two sets with $X=F\cap F'$ of size at most $(r-2)(t-1)$. Then one can partition $X$ to $X_1,X_2,\dots,X_{r-2}$ with $|X_i|\le t-1$ for $i$. As $\tau(\cF)=t$, none of the $X_i$'s is a cover of $\cF$, so  for all $i$ there exists $F_i\in \cF$ disjoint with $X_i$. But then $F\cap F'\cap \bigcap_{i=1}^{r-2}F_i=\emptyset$ which contradicts the $r$-wise intersecting property of $\cF$.
\end{proof}

Proposition \ref{rwisetau} implies that if $r>k$, then we have $\tau(\cF)=1$. Indeed, if $\tau(\cF)\ge 2$, then a $k$-uniform $r$-wise intersecting family $\cF$ must be $k$-intersecting, which means that $\cF$ must consist of a single $k$-set contradicting $\tau(\cF)\ge 2$.

\medskip

The next two propositions provide some very weak bounds on the size of intersecting families with respect to their covering number. Many stronger results are known (see e.g. \cite{F2,FK, FOT, FOT2,FOT3, HM}), but even those would not yield linear thresholds for Theorems \ref{tria} and \ref{gener}. On the other hand, the proofs of the propositions are so easy and short that we can include them for sake of self-containedness.

\begin{proposition}\label{tau}
If $\cF\subseteq \binom{[n]}{k}$ is an intersecting family with $\tau(\cF)\ge t$, then $|\cF|\le k^t\binom{n-t}{k-t}$.
\end{proposition}

\begin{proof}
For any subset $S\subset [n]$ let the degree of $S$ be $d_\cF(S):=|\{F\in \cF:S\subseteq F\}|$ and we write $D_m$ for the maximum degree over all sets of size $m$. Observe that as long as $m<t$, we have $D_m\le k\cdot D_{m+1}$. Indeed, for any $m$-subset $S$ there exists $F\in \cF$ with $S\cap F=\emptyset$, so every $S\subseteq F'\in \cF$ must meet $F$ and thus $d_\cF(S)\le \sum_{x\in F}d_\cF(S\cup \{x\})$. Clearly, we have $D_t\le \binom{n-t}{k-t}$ and $D_0=|\cF|$. This finishes the proof.
\end{proof}

Not necessarily distinct families $\cF_1,\cF_2,\dots,\cF_m$ of sets are said to be \textit{cross-$t$-intersecting} if for any $F_1\in \cF_1,F_2\in\cF_2,\dots, F_m\in \cF_m$ we have $|\cap_{i=1}^mF_i|\ge t$.

\begin{proposition}\label{cross}
If $\cF_1,\cF_2,\dots,\cF_m\subseteq \binom{[n]}{k}$ are cross-$t$-intersecting with $\tau(\cF_i)\ge 2$ for all $i=1,2,\dots,m$, then $|\cF_i|\le \frac{k}{t}\binom{k}{t}\binom{n-t-1}{k-t-1}$ for all $i=1,2,\dots,m$.
\end{proposition}

\begin{proof}
Observe first that $\tau(\cF_i)\ge 2$ implies that all $\cF_i$s are non-empty. Without loss of generality, we bound the size of $\cF_1$. Fixing $G\in \cF_2$ we have $|\cF_1|\le \frac{1}{t}\sum_{x\in G}|\{F\in \cF_1:x\in F\}|$. As $\tau(\cF_2)\ge 2$, for every $x\in G$ there exists $G_x\in \cF_2$ with $x\notin G_x$. So any $F\in \{F\in \cF_1:x\in F\}$ must meet $G_x$ in at least $t$ elements, thus $|\{F\in \cF_1:x\in F\}|\le \binom{k}{t}\binom{n-t-1}{k-t-1}$.
\end{proof}

\section{Proofs}

For any family $\cF$ we will  write $\cN(\Delta_{r+1},\cF)$ to denote the number of $(r+1)$-triangles in $\cF$.

\begin{proof}[Proof of Theorem \ref{tria}]
Let $\cF\subseteq \binom{[n]}{k}$ be an intersecting family. We can assume that  $\cF$ is maximal as adding sets to $\cF$ can only increase the number of triangles. Therefore we can assume that any $k$-set $F$ that contains a cover $C\in \cC_\cF$ belongs to $\cF$. (Note that this is true only for maximal intersecting families, but not for maximal $r$-wise intersecting families with $r\ge 3$.)

\begin{claim}\label{cover}
If $n\ge 2k$ and $C,C'\in \cC_\cF$, then $C\cap C'\neq \emptyset$.
\end{claim}

\begin{proof}[Proof of Claim]
Suppose not. Then, as $n\ge 2k$, there exist $k$-sets $F,F'$ with $C\subseteq F, C'\subseteq F'$ and $F\cap F'=\emptyset$. This contradicts the intersecting property of $\cF$.
\end{proof}

We consider several cases according to $\tau(\cF)$.

\medskip

\textsc{Case I} $\tau(\cF)=1$.

Then the unique element of a singleton cover belongs to all sets of $\cF$, so $\cF$ does not contain any triangles.

\medskip

\textsc{Case II} $\tau(\cF)\ge 3$.

Then by Proposition \ref{tau}, we have $|\cF|\le k^3\binom{n-3}{k-3}$, and thus $\cN(\Delta_3,\cF)\le \binom{|F|}{3}\le k^9\binom{n-3}{k-3}^3$ which, if $n\ge 2k^{4}$, is smaller than $\frac{1}{2}\binom{n-3}{k-2}^3$, and thus than the value of $n_{3,k}$ by Proposition \ref{ordmag}.

\medskip 

\textsc{Case III} $\tau(\cF)=2$

By Claim \ref{cover}, $\cC_\cF\cap \binom{[n]}{2}$ is either a graph star $S_\ell$ with $\ell$ leaves or a graph triangle. We consider first the case when we have a star with center $c$.

Suppose first $\ell \ge 3$. Then any triangle must contain a set $F$ that does not contain $c$. As $\ell\ge 3$ and $F$ must contain all leaves, there are at most $\binom{n-\ell-1}{k-\ell}\le \binom{n-4}{k-3}$ such sets. Thus the number of triangles is at most $\binom{n-4}{k-3}\cdot \binom{|\cF|}{2}\le k^4\binom{n-4}{k-3}\binom{n-2}{k-2}^2$ by Proposition \ref{tau}. If $n\ge 4k^5$, then this is smaller than $\frac{1}{2}\binom{n-3}{k-2}^3$, and thus than $n_{3,k}$.

Suppose next $\ell=2$ with $\cC_\cF\cap \binom{[n]}{2}=\{\{c,x\},\{c,y\}\}$. As $\{x,y\}$ is not a cover, there must be a set $F'\in \cF$ with $x,y\notin F'$. On the other hand every $F\in \cF$ with $c\notin F$ must contain both $x$ and $y$ (as $\{c,x\}$ and $\{c,y\}$ are covers) and meet $F'$, so their number is not more than $k\binom{n-3}{k-3}$. So the number of triangles in $\cF$ is at most $k\binom{n-3}{k-3}\binom{|\cF|}{2}\le k^5\binom{n-3}{k-3}\binom{n-2}{k-2}^2$, which is smaller than $\frac{1}{2}\binom{n-3}{k-2}^3$ if $n\ge 4k^6$.

Finally, suppose that $\ell=1$, i.e. $\{x,y\}$ is a unique cover set of $\cF$ of size 2. We try to bound the size of $\cF_x=\{F\in \cF: x\in F,y\notin F\}$ and $\cF_y=\{F\in \cF: x\notin F,y\in F\}$. As $\{x\}$ and $\{y\}$ are not covers but $\{x,y\}$ is, they are not empty. Fix $F_y\in \cF_y$ and consider $F\in \cF_x$. All such $F$ must meet $F_y$ in some $y'\neq y$. There are $k-1$ choices for $y'$ and for each $y'$, as $\{x,y'\}$ is not a cover, there exists $F_{y'}\in \cF$ with $x,y'\notin F_{y'}$. So every $F\ni x,y'$ must meet $F_{y'}$, so there are at most $k\binom{n-3}{k-3}$ such sets. Summing over all $y'$, we obtain $|\cF_x|\le (k-1)k\binom{n-3}{k-3}$. An identical argument yields $|\cF_y|\le (k-1)k\binom{n-3}{k-3}$. Observe that every triangle must contain a set both from $\cF_x$ and $\cF_y$ - otherwise its three sets would have a common element as $\{x,y\}$ is a cover. Therefore the number of triangles in $\cF$ is at most $|\cF_x|\cdot |\cF_y|\cdot |\cF|\le k^6\binom{n-3}{k-3}^2\binom{n-2}{k-2}$ which is smaller than $\frac{1}{2}\binom{n-3}{k-2}^3\le n_{3,k}$ if $n\ge 2k^4$.

\smallskip

In all cases, we obtained that $\cN(\Delta_3,\cF)<n_{3,k}$, so we must have that $\cC_\cF\cap \binom{[n]}{2}$ is a graph triangle $\{x,y,z\}$. So every $F\in\cF$ must contain at least 2 of $x,y,z$ and thus $\cF\subseteq \cF_{\{x,y,z\},k}$. In order to have as many triangles as possible, $\cF$ must contain all sets in $\cF'_{x,y,z}$. This finishes the proof of Theorem \ref{tria}.
\end{proof}

\medskip

We continue with the proof of the general result.

\begin{proof}[Proof of Theorem \ref{gener}] Let $\cF\subseteq \binom{[n]}{k}$ be an $r$-wise intersecting family. We start with a lemma on the covering number of $\cF$.

\begin{lemma}\label{tau3}
If $\cF\subseteq \binom{[n]}{k}$ is $r$-wise intersecting with $\tau(\cF)\ge 3$, then $\cN(\Delta_{r+1},\cF)<n_{r+1,k}$.
\end{lemma}

\begin{proof}
By Proposition \ref{rwisetau}, $\cF$ is $(2r-3)$-intersecting with $\tau(\cF)>1$, so applying Proposition \ref{cross} with $\cF_i=\cF$ we obtain $|\cF|\le k\binom{k}{2r-3}\binom{n-2r+2}{k-2r+2}$. Clearly, we have $$\cN(\Delta_{r+1},\cF)\le \binom{|\cF|}{r+1}\le \left[k^{2r-2}\binom{n-2r+2}{k-2r+2}\right]^{r+1}<\frac{1}{2}\binom{n-r-1}{k-r}^{r+1}\le n_{r+1,k},$$ where the strict inequality uses $r\ge 3$ and $n\ge 2k^{2r-1}$, and the last inequality is Proposition \ref{ordmag}.
\end{proof}

If $\tau(\cF)=1$, then $\cF$ does not contain any $(r+1)$-triangles, so by Lemma \ref{tau3}, we can assume $\tau(\cF)=2$. Let $G$ be the graph with vertex set $[n$] and edge set $\cC_\cF\cap \binom{[n]}{2}$. By Proposition \ref{rwisetau}, we know that $\cF$ is $(r-1)$-intersecting with $\tau(\cF)=2$, therefore, by Proposition \ref{cross} with $\cF_i=\cF$, we obtain $|\cF|\le k^r\binom{n-r}{k-r}$. We will use this bound several times below. 

\begin{lemma}\label{P3}
If $G$ contains an induced path on 3 vertices, then $\cN(\Delta_{r+1},\cF)<n_{r+1,k}$.
\end{lemma}

\begin{proof}
Suppose $xc,yc$ are edges of $G$ but $xy$ is not. Then any $(r+1)$-triangle must contain a set $F\in \cF$ with $c\notin F$. As $xc$ and $yc$ are covers of $\cF$, all such sets must contain both $x$ and $y$. Furthermore, as $xy$ is not a cover of $\cF$, there must exist $F'\in F$ with $x,y\notin F'$. By Proposition \ref{rwisetau}, we have $|F\cap F'|\ge r-1$, so the number of possible $F$s is at most $\binom{k}{r-1}\binom{n-r-1}{k-r-1}$. Therefore, the number of $(r+1)$-triangles in $\cF$ is at most $$\binom{k}{r-1}\binom{n-r-1}{k-r-1}|\cF|^r\le k^{r^2+r-1}\binom{n-r-1}{k-r-1}\binom{n-r}{k-r}^r<\frac{1}{2}\binom{n-r-1}{k-r}^{r+1}\le n_{r+1,k},$$
where we used $n\ge 4k^{r^2+r}$ for the strict inequality.
\end{proof}

\begin{lemma}\label{maxdeg}
If $G$ has maximum degree larger than $r$, then $\cN(\Delta_{r+1},\cF)<n_{r+1,k}$.
\end{lemma}

\begin{proof}
Suppose $c$ has degree at least $r+1$ in $G$. Any $(r+1)$-triangle must contain a set $F\in \cF$ with $c\notin F$. All such $F$ must contain all neighbors of $c$ as any neighbor together with $c$ form a cover of $\cF$. Therefore, the number of such sets is at most $\binom{n-r-2}{k-r-1}$. Thus the number of $(r+1)$-triangles in $\cF$ is at most 
$$\binom{n-r-2}{k-r-1}|\cF|^r\le k^{r^2}\binom{n-r-2}{k-r-1}\binom{n-r}{k-r}^r<\frac{1}{2}\binom{n-r-1}{k-r}^{r+1}\le n_{r+1,k},$$
where we used $n\ge 4k^{r^2+1}$ for the strict inequality.
\end{proof}

Lemma \ref{P3} yields that all components of $G$ are cliques, and Lemma \ref{maxdeg} shows that no component has size more than $r+1$.

\begin{lemma}
If $G$ contains a component of size smaller than $r+1$ that is not an isolated vertex, then $\cN(\Delta_{r+1},\cF)<n_{r+1,k}$.
\end{lemma}

\begin{proof}
Without loss of generality we can assume that $1,2,\dots,m$ form a clique in $G$ with $m\le r$. Therefore every set $F\in \cF$ meets $[m]$ in at least $m-1$ vertices. Also, every $(r+1)$-triangle must contain a set $F_i$ with $i \notin F_i$ for all $i=1,2,\dots,m$. Let $\cF_i$ denote $\{F\in \cF:i\notin F\}$ and $\cF'_i=\{F\setminus [m]:F\in \cF_i\}$, $\cF'=\cup_{i=1}^m\cF'_i$. Then $|\cF_i|=|\cF'_i|$ and $\cF'_i$, $\cF'$ are $(k-m+1)$-uniform.

Observe first that we have $\tau(\cF'_i)\ge 2$ for all $i$. Indeed, if $\{x\}$ was a cover of $\cF'_i$, then $x\notin [m]$ and $\{i,x\}$ would be a cover of $\cF$, so an edge in $G$, but $x$ and $i$ are in different components of $G$.

Next we claim that the $\cF'_i$s are cross-$(r-m+1)$-intersecting (as $m\le r$, this a positive number). Indeed, if for some $F'_1\in \cF'_1,F'_2\in\cF'_2,\dots,F'_m\in \cF'_m$ we had $|\cap_{i=1}^mF'_i|\le r-m$, then for any $x\in \cap_{i=1}^mF'_i=:X$ there exists $F_x\in \cF$ with $x\notin F_x$, because $\tau(\cF)\ge 2$. Then $\cap_{i=1}^mF_i\cap_{x\in X}F_x =\emptyset$, where $F_i\in \cF_i$ with $F_i\setminus [m]=F'_i$. This contradicts the $r$-wise intersecting property of $\cF$ and thus proves our claim.

Applying Proposition \ref{cross}, we obtain  $|\cF_i|=|\cF'_i| \le \frac{k-m+1}{r-m+1}\binom{k-m+1}{r-m+1}\binom{n-r+m-2}{k-r-1}\le k^{r-m+2}\binom{n-r+m-2}{k-r-1}$. 
Observe that $$\cN(\Delta_{r+1},\cF)\le |\cF|^{r+1-m}\cdot \prod_{i=1}^m|\cF_i|\le \left(k^r\binom{n-r}{k-r}\right)^{r+1-m}\cdot \left[k^{r-m+2}\binom{n-r+m-2}{k-r-1}\right]^m$$

This last expression decreases with $m$, so takes its maximum at $m=2$. So the number of $(r+1)$-triangles in $\cF$ is at most $$k^{r(r+1)}\binom{n-r}{k-r}^{r-1}\binom{n-r}{k-r-1}^2<\frac{1}{2}\binom{n-r-1}{k-r}^{r+1}\le n_{r+1,k}$$
if $n\ge 4k^{\frac{r(r+1)+2}{2}}$ holds.
\end{proof}
We obtained that unless $G$ consists only of cliques of size $r+1$, the number of $(r+1)$-triangles in $\cF$ is strictly smaller than $n_{r+1,k}$. On the other hand, if $X$ is the vertex set of a clique of size $r+1$ in $G$, then $\cF \subseteq \cF_{X,k}$ as every $F\in \cF$ must meet $X$ in at least $|X|-1$ elements to intersect all edges of the clique. To contain as many $(r+1)$-triangles as possible, we must have $\cF'_{X,k}\subseteq \cF$. This finishes the proof of Theorem \ref{gener}.
\end{proof}

\medskip

It would be interesting to obtain a different proof of Theorems \ref{tria} and \ref{gener} that results in a smaller threshold on $n$ with respect to $k$ and $r$. Also, the non-uniform version of these theorems are of interest. We conjecture that an analogous statement should hold.

\begin{conjecture}
 For any $r\ge 2$ there exists $n_0=n_0(r)$ such that if $\cF\subseteq 2^{[n]}$ is $r$-wise intersecting with $n\ge n_0$, then the number of $(r+1)$-triangles in $\cF$ is at most $\cN(\Delta_{r+1},\cF_X)$, where $\cF_X=\{F\subseteq [n]: |F\cap X|\ge |X|-1\}$ for some $(r+1)$-set $X$. 
\end{conjecture}

\bigskip

\noindent\textbf{Acknowledgement.} We would like to thank Stijn Cambie for pointing out some calculation errors in the first version of our manuscript.

\end{document}